\documentclass[12pt,reqno]{amsart}
\usepackage{amsmath, amsfonts, amssymb, amsthm,mathrsfs}
\def\Z{\mathbb Z}
\def\C{\mathbb C}

\def\T{\mathbb T}

\def\cX{{\mathcal X}}
\def\cY{{\mathcal Y}}
\def\cZ{{\mathcal Z}}

\def\1{{\bf 1}}

\def\pmod #1{\ ({\rm{mod}}\ #1)}

\theoremstyle{plain}
\newtheorem{theorem}{Theorem}

\newtheorem{lemma}{Lemma}

\theoremstyle{definition}
\newtheorem*{acknowledgment}{Acknowledgments}
\theoremstyle{remark}
\newtheorem*{Rem}{Remark}

\begin{document}
\title
{On the Balog-Ruzsa Theorem in short intervals}

\author{Yu-Chen Sun}
\address {Department of Mathematics and Statistics, University of Turku, 20014 Turku, Finland}
\email{yuchensun93@163.com}

\keywords{M\"obius function, $r$-free numbers, $L_1$ norm, exponential sum}
\subjclass[2010]{}
 \begin{abstract}
 In this paper we give a short interval version of the Balog-Ruzsa theorem concerning bounds for the $L_1$ norm of the exponential sum over $r$-free numbers. As an application, we give a lower bound for the $L_1$ norm of the exponential sum defined with the M\"obius function. Namely we show that
 $$
 \int_{\T} \left|\sum_{|n-N|<H} \mu(n)e(n \alpha)\right| d \alpha \gg H^{\frac{1}{6}}
 $$
 when $H \gg N^{\frac{9}{17}  + \varepsilon}$.
\end{abstract}
\maketitle

\section{Introduction}\label{intro}
\setcounter{lemma}{0} \setcounter{theorem}{0}
\setcounter{equation}{0}

For an integer $r \geq 2$, we say an integer $n$ is $r$-free if it has no factor $d>1$ which is an $r$-th power. In 1998 Br\"udern, Granville, Perelli, Vaughan and Wooley \cite{BGPVW98} studied bounds of the $L_1$ norm for the exponential sum over $r$-free numbers and gave the first nontrivial bounds.  In 2001 Balog and Ruzsa \cite{BR01} improved on the bounds and gave the best possible bound for the $L_1$ norm of the exponential sum over $r$-free numbers.

Let $r \geq 2$ be fixed and $a_n$ be the characteristic function of the $r$-free integers, that is
 \begin{equation}\label{defan}
 a_n:=
 \begin{cases}
 1,& \text{if } n \text{ is } r\text{-free},\\
 0,& \text{otherwise}
 \end{cases}
 \end{equation}
Balog and Ruzsa \cite{BR01} proved that
\begin{theorem}[Balog-Ruzsa]
Let $N \geq 2$. Then
$$
N^{\frac{1}{r+1}} \ll \int_{\T} \left|\sum_{n=1}^{N} a_n e(n \alpha)\right| d \alpha \ll N^{\frac{1}{r+1}}
$$
\end{theorem}
The most interesting case is $r=2$, which is related to the squarefree numbers and the M\"obius function $\mu(n)$. As a corollary their theorem implies that
\begin{theorem}
Let $N \geq 2$. Then
$$
\int_{\T} \left|\sum_{n=1}^{N} \mu(n) e(n \alpha)\right| d \alpha \gg N^{\frac{1}{6}}.
$$
\end{theorem}

In this paper, we will give a short interval version of the Balog-Rusza theorem. We are most interested in the lower bound of the $L_1$ norm of the exponential sum over squarefree numbers and the exponential sum of the M\"obius function, and obtain the following results.
\begin{theorem}\label{Mobius}
Let $\epsilon>0$ and $N \geq H \geq N^{\frac{9}{17} + \varepsilon}$. Then
$$
 \int_{\T} \left|\sum_{|n-N|<H} \mu(n)e(n \alpha)\right| d \alpha \gg H^{\frac{1}{6}}.
$$
\end{theorem}
Theorem \ref{Mobius} is essentially a corollary of the following theorem.
\begin{theorem}\label{sqfre}
Let $\epsilon>0$ and $N \geq H \geq N^{\frac{9}{17} + \varepsilon}$. Then
\begin{equation}\label{main}
 \int_{\T} \left|\sum_{|n-N|<H} \mu^2(n) e(n \alpha)\right| d \alpha \gg H^{\frac{1}{3}}.
\end{equation}
\end{theorem}

We will prove Theorem \ref{sqfre} in Section \ref{pflow} and show how Theorem \ref{sqfre} implies Theorem \ref{Mobius} in the end of this section.
As a short interval analog of Balog-Ruzsa theorem we also get the following theorem
\begin{theorem}\label{rfree}
Let $\epsilon > 0$ and 
\begin{equation}\label{lengthHrfree}
N \geq N \geq
\begin{cases}
 N^{\frac{18}{29} + \epsilon} & \text{if }r=2,\\
 N^{\frac{r+1}{2r} + \epsilon} & \text{otherwise}.
\end{cases}
\end{equation}
Then
$$
H^{\frac{1}{r+1}} \ll \int_{\T} \left|\sum_{|n-N|\leq H} a_n e(n \alpha)\right| d \alpha \ll H^{\frac{1}{r+1}}.
$$
\end{theorem}
\begin{Rem}
The result for $r \geq 3$ follows from Balog-Ruzsa's arguments. However, for $r=2$, our results are much stronger than the trivial result $\frac{r+1}{2r} = \frac{3}{4} = 0.75$. For $r=2$, we note that our exponent is $\frac{9}{17} =0.529\dots$ in lower bound case, and $\frac{18}{29} = 0.620\dots$ in upper bound case.

Besides, for $r=2$, Theorem \ref{sqfre} gives the lower bound (\ref{main}) for shorter intervals than Theorem \ref{rfree}. Thus from Theorem \ref{rfree} we can not get Theorem \ref{Mobius}.
\end{Rem}

Balog-Ruzsa's argument utilizes the Fej\'er Kernel which has many good properties (see \cite{SS03}) and has many applications in number theory (see \cite{M73}).

Let $e(x) := e^{2 \pi i x}$. The Fej\'er Kernel $F(\alpha)$ is the Ces\`aro mean of the Dirichlet Kernel $D_{N}(\alpha) = \sum_{|n|\leq N} e(n \alpha)$ and is defined by
\begin{equation}\label{Fejer}
F(\alpha):= \sum_{|n| \leq N}\left(1 - \frac{|n|}{N}\right)e(n \alpha) = \frac{\sin^2 (\pi N \alpha)}{N \sin^2 (\pi \alpha)} \leq \min\left\{N, \frac{1}{N \|\alpha\|^2}\right\}.
\end{equation}
For the short interval case, we need to introduce a short interval version of the Fej\'er Kernel which is motivated by the perspective of the Fourier analysis. For convenience we first introduce some notions in Fourier analysis.

Let $A \subset \Z$ be finite. For $f: A \to \C$, we define the Fourier transform $\hat{f}: {\mathbb T} \to \C$ of $f$ by
$$
\hat{f} (\alpha) = \sum_{n \in A} f(n) e(n \alpha).
$$
If $f,g: A \to \C$, we define the convolution of $f$ and $g$ to be
$$
f \ast g(n) = \sum_{x+y=n} f(x)g(y),
$$
A basic property involving Fourier transform and convolution is that
$$
\widehat{f \ast g}(\alpha) = \hat{f}(\alpha) \hat{g}(\alpha).
$$
Similarly, we define the convolution of the Fourier transform as
\begin{equation}\label{FourCov}
\hat{f} \ast \hat{g}(\alpha) = \int_{\T} \hat{f}(\alpha - \beta)\hat{g}(\beta) d \beta = \widehat{fg}(\alpha).
\end{equation}

 Now the Fej\'er Kernel $F(\alpha)$ can be described through the ``Fourier language''.
\begin{align}
F(\alpha) & = \frac{1}{N}\widehat{\1_{[-\frac{N}{2} , \frac{N}{2})} * \1_{(-\frac{N}{2} , \frac{N}{2}]}} = \frac{1}{N}\hat{\1}_{[-\frac{N}{2} , \frac{N}{2})} \bar{\hat{\1}}_{[-\frac{N}{2} , \frac{N}{2})}\notag \\
& = \frac{1}{N}\left|\sum_{n \in (-\frac{N}{2} , \frac{N}{2}]} e(n \alpha)\right|^2 =  \frac{\sin^2 (\pi N \alpha)}{N \sin^2 (\pi \alpha)} \ll \min\left\{N , \frac{1}{N \|\alpha\|^2}\right\}. \label{FejFour}
\end{align}
Let us introduce some properties of the difference of two Fej\'er Kernels, which will be used in Section \ref{pfupp} to prove the upper bound case of Theorem \ref{rfree}.
By the difference of the two Fej\'er Kernels, we mean
\begin{equation}\label{FejDifFour}
\frac{1}{K} \left(\widehat{\1_{(-\frac{N+K}{2} , \frac{N+K}{2}]} * \1_{[-\frac{N+K}{2} , \frac{N+K}{2})}} - \widehat{\1_{(-\frac{N}{2} , \frac{N}{2}]} * \1_{[-\frac{N}{2} , \frac{N}{2})}}\right).
\end{equation}
From (\ref{Fejer}) and (\ref{FejFour}) one can easily get
\begin{align}
& \frac{1}{K} \left(\widehat{\1_{(-\frac{N+K}{2} , \frac{N+K}{2}]} * \1_{[-\frac{N+K}{2} , \frac{N+K}{2})}} - \widehat{\1_{(-\frac{N}{2} , \frac{N}{2}]} * \1_{[-\frac{N}{2} , \frac{N}{2})}}\right) \notag\\
= & \sum_{|n| \leq N+K} \min \left\{1, \frac{N+K-|n|}{K} \right\}e(n \alpha) = \frac{\sin^2(\pi(N+K)\alpha) - \sin^2(\pi K \alpha)}{K \sin^2 (\pi \alpha)} \notag\\
\label{DfFejerupp}= &  \frac{\sin(\pi(2N+K)\alpha)\sin(\pi K \alpha) }{K \sin^2 (\pi \alpha)} \ll \min\left\{N+K, \frac{1}{\|\alpha\|}, \frac{1}{K \|\alpha\|^2}\right\}
\end{align}

Now we define our short interval version of the Fej\'er Kernel $F_{H}(\alpha)$, For convenience we can assume that both $N$ and $H$ are even integers. We write
\begin{equation}\label{FL1}
F_{H}(\alpha) = \frac{1}{H}\widehat{\1_{(\frac{N-H}{2}, \frac{N+H}{2}]} \ast \1_{[\frac{N-H}{2},\frac{N+H}{2})}}(\alpha).
\end{equation}
Similar to (\ref{FejFour}), our $F_{H}(\alpha)$ also has two explicit formulas and an upper bound. Namely,
\begin{equation}\label{FEx}
F_{H}(\alpha) = e(N \alpha)\sum_{|h|\leq H}\left(1 -\frac{h}{H} \right)e(h \alpha)= \sum_{|n-N|\leq H}\left(1 -\frac{|n-N|}{H} \right)e(n \alpha)
\end{equation}
and then
\begin{align}
F_{H}(\alpha) & = \frac{1}{H} e(N \alpha) \frac{\sin^2 (\pi H \alpha)}{\sin^2 (\pi \alpha)} \label{Flow} \\
& \ll \min\left\{H, \frac{1}{H \|\alpha\|^2}\right\} \label{Fupp}.
\end{align}

\begin{Rem}
By using this $F_H$ we can reserve almost all properties from Fej\'er Kernel, but from (\ref{Flow}) we can see that this $F_H$ is not always real, which might be the only property it loses from Fej\'er Kernel.
\end{Rem}

Recall the definition of $a_n$ from (\ref{defan}) and let $g_1(\alpha) = \sum_{|n-N| \leq H} a_n e(n \alpha)$.
Before proving our results, we should notice that $\|g_1\|_{L_1}$ and $\|F_{H} \ast g_1\|_{L_1}$ are comparable. Namely, we have
\begin{equation}\label{L1FHg1}
\|F_H \ast g_1\|_{L_1} \leq \int_{\T} \int_{\T} |F_H (\alpha - \beta)| |g_1(\beta)| d \beta d \alpha \leq \|g_1\|_{L_1} \|F_H\|_{L_1} \ll \|g_1\|_{L_1}.
\end{equation}
The last inequality immediately follows from (\ref{Flow}). Note that $a_n = \mu^2(n)$ when $r=2$ in (\ref{defan}). Thus, for Theorem \ref{sqfre} and the lower bound case of Theorem \ref{rfree}, it suffices to find a lower bound for $\|F_{H} \ast g_1\|_{L_1}$. The above idea can also be used to deduce that Theorem \ref{Mobius} impies Theorem \ref{sqfre}.

\begin{proof}[Proof that Theorem \ref{sqfre} implies Theorem \ref{Mobius}]
Define
$$
\mu_{H}(n) :=\mu(n)\1_{|n-N|<H}
$$
Then we have
$$
\widehat{\mu_H}(\alpha) = \sum_{|n-N|<H} \mu(n) e(n \alpha),
$$
and by (\ref{FourCov}) we have
$$
\widehat{\mu^2_{H}}(\alpha) = \widehat{\mu_{H}}*\widehat{\mu_{H}}(\alpha).
$$

Similar to (\ref{L1FHg1}), we have
$$
\|\widehat{\mu^2_{H}}\|_{L_1} = \|\widehat{\mu_{H}} * \widehat{\mu_{H}}\|_{L_1} \ll \|\widehat{\mu_{H}}\|_{L_1}^2.
$$
Now by Theorem \ref{sqfre}, $\|\widehat{\mu^2_{H}}\|_{L_1} \gg H^{1/3}$, and the claim follows.
\end{proof}

The rest of the paper is organized as follows. In Section \ref{auxiliary_lemmas}, we collect some auxiliary lemmas concerning the Riemann zeta function, Dirichlet polynomials, and van der Corput bounds. In Section \ref{sect_key}, we will prove a key lemma, which improves Lemma 1 of \cite{BR01}, using an analytic approach and van der Corput bounds. Then in Sections \ref{pflow} and \ref{pfupp}, we will follow Balog's and Ruzsa's arguments to prove our lower bound and upper bound results.

\begin{acknowledgment}
The author is grateful to his supervisor Kaisa Matom\"aki for many useful discussions, reading the paper carefully and giving a lot of helpful comments, and to Hao Pan and Joni Ter\"av\"ainen for drawing his attention to the van der Corput methods. During the work the author was supported by EDUFI funding.
\end{acknowledgment}

\section{Some auxiliary lemmata}\label{auxiliary_lemmas}
\setcounter{lemma}{0} \setcounter{theorem}{0}
\setcounter{equation}{0}

\begin{lemma}[Perron's Formula]\label{Perron}
For $\Re(s)> 1$, let
$$
f(s)= \sum_{n=1}^{\infty} a_n n^{-s},
$$
where $a_n = O(\phi(n))$, $\phi(n)$ is non-decreasing. Let $\alpha > 0$ and as $\sigma \to 1^{+}$,
$$
\sum_{n=1}^{\infty} a_n n^{-\sigma} = O\left( \frac{1}{(\sigma-1)^{\alpha}} \right).
$$
Then if $c>1$ and $x$ is not an integer, we have
\begin{align*}
\sum_{n \leq x}a_n = & \frac{1}{2 \pi i} \int_{c - iT}^{c + iT} f(s) \frac{x^s}{s} ds + O\left(\frac{x^c}{T(c-1)^{\alpha}}\right) \\
& + O\left(\frac{\phi(2x) x \log x}{T}\right) + O\left(\frac{\phi(N)x}{T\|x\|}\right).
\end{align*}
\end{lemma}
\begin{proof}
See \cite[Lemma 3.12]{T86}.
\end{proof}

\begin{lemma}\label{zeta_ubd}
Let  $\epsilon>0$. Suppose that $\frac{1}{2} \leq \sigma \leq 1 + \varepsilon $ and $t \geq 1$. Then
$$
\zeta(\sigma + it) = O(t^{\frac{1}{3}(1 - \sigma) + \epsilon}).
$$
\end{lemma}
\begin{proof}
See \cite[Chapter 5, in particular, (5.12), Theorem 5.5 and the convexity of $\mu(\sigma)$.]{T86} 
\end{proof}

\begin{lemma}\label{mean_zeta}
For $T \geq 2$, we have
$$
\int_{-T}^{T}\left|\zeta\left(\frac{1}{2} + it\right)\right|^2 \ll T \log T.
$$
\end{lemma}
\begin{proof}
See \cite[Theorem 7.2(A)]{T86}.
\end{proof}

\begin{lemma}\label{mean_Dpoly}
For $T, N \geq 2$ and any complex numbers $a_n$ we have
$$
\int_{0}^{T} \left|\sum_{0 \leq n \leq N} a_n n^{it}\right|^2 dt = (T + O(N))\sum_{0 \leq n \leq N}|a_n|^2
$$
\end{lemma}
\begin{proof}
See \cite[Theorem 9.1]{IK04}.
\end{proof}

\begin{lemma}[van der Corput bounds]\label{VDinequality}
Suppose that $(p,q)$ is an exponent pair defined by \cite[(8.57) and (8.58)]{IK04}, $f(x)$ behaves like a monomial so that
$$
|f^{(j)}(x)| \asymp_{j} FM^{-j}
$$
for every $j \geq 0$, any $x \in [M,2M]$ and some $F \geq M$. We define $\psi(x):= \{x\} -\frac{1}{2}$ and let $I = [a,b] \subset [M,2M]$. Then
$$
\sum_{m \in I}\psi(f(m)) \ll F^{\frac{p}{p+1}}M^{\frac{1+2q}{2(p+1)}+\epsilon}
$$
for any $\epsilon > 0$.
\end{lemma}
\begin{proof}
By \cite[Theorem A.6]{GK91}, there exist coefficients $\beta(r) \ll |r|^{-1}$ such that
$$
\psi(x) \leq R^{-1} + \sum_{1 \leq |r| \leq R} \beta(r)e(rx).
$$
Thus,
$$
\sum_{m \in I}\psi(f(m)) \ll MR^{-1} + \sum_{1 \leq r \leq R}\frac{1}{r} \left|\sum_{m \in I}e(rf(m))\right|.
$$
By \cite[(8.58)]{IK04}, we obtain that
\begin{align*}
\sum_{m \in I}\psi(f(m))& \ll MR^{-1} + \sum_{1 \leq r \leq R}\frac{1}{r} (rFM^{-1})^{p}M^{q+\frac{1}{2}+\epsilon'}\\
& \ll MR^{-1} + F^{p}M^{q-p+\frac{1}{2}+\epsilon'}R^p.
\end{align*}
We choose $R$ so that $R^{p+1} = M^{\frac{1}{2} + p - q - \epsilon'}F^{-p}$, and the claim follows.
\end{proof}

\begin{Rem}
We will only need the lemma for the exponent pair $(p,q) = \left(\frac{2}{7},\frac{1}{14}\right)$
\end{Rem}

\section{Key lemmas}\label{sect_key}
\setcounter{lemma}{0} \setcounter{theorem}{0}
\setcounter{equation}{0}

We recall that $a_n = \sum_{d^r \mid n} \mu(d)$. Let $ z > y \geq 1$ be any real numbers. For $n \neq 0$ we define $c_{n} = c_{n}(y,z)$ as a middle part of the above sum, namely
\begin{equation}\label{cnyz}
c_{n} = \sum_{\substack{d^r \mid n \\ y < d \leq z}} \mu(d).
\end{equation}
\begin{lemma}\label{key}
For any $1 \leq K < N$ and $1 \leq y < z$ we have
\begin{equation}\label{r>2}
\sum_{N-K < n \leq N}|c_n|^2 \ll K y^{1-r} + N^{1/r + \epsilon}.
\end{equation}
Furthermore, for $r=2$ and $y \leq  K^{1/2 - \epsilon}$, we have
\begin{equation}\label{r=2}
\sum_{N-K < n \leq N}|c_n|^2 \ll K y^{-1} + N^{\frac{12}{29} + \epsilon}y^{-\frac{10}{29}}.
\end{equation}
\end{lemma}

\begin{proof}
 In fact (\ref{r>2}) is  \cite[Lemma 1]{BR01}, so we just need to prove (\ref{r=2}). For $r=2$, we have
\begin{align}
& \sum_{N-K < n \leq N}|c_n|^2 \leq \sum_{N-K < n \leq N} (\sum_{\substack{d^2 \mid n \\ y < d \leq z}} 1)^2 = \sum_{N-K < n \leq N} \sum_{\substack{d_1^2, d_2^2 \mid n \\ y < d_i\leq z}}1 \notag\\
\notag \ll & \sum_{N-K < n \leq N} \sum_{h^2 \mid n} \sum_{\substack{(hd_1')^2,(hd_2')^2 \mid n \\ (d_1',d_2')=1 \\ y < hd'_i \leq z}}1\\
\label{ubdcn} \ll & \sum_{N-K < n \leq N} \sum_{\substack{n = h^2 d_1^2 d_2^2 a \\ y < hd_i \leq z}}1
\end{align}

 So we only need to show that (\ref{ubdcn}) is $\ll K y^{-1} + N^{\frac{12}{29} + \epsilon}y^{-\frac{10}{29}}$. We first split $d_i, h$ into dyadic ranges, so we consider the sum over $d_i \sim D_i$ and $h \sim H$ with $y \ll HD_i \ll z$ and then add those sums later.

 Before proving the upper bound for (\ref{ubdcn}), let us explain how we use two different techniques depending on the size of $D_i$ and $H$. Without loss of generality, we can assume that $D_2 \gg D_1$ and also notice that $D_1H \gg y$. If $D_2$ is large, we can use an analytic approach to give a good upper bound (see (\ref{uppAnalytic})). If $D_2$ is small, we can use the hyperbolic trick and the van der Corput method.

For convenience we can assume that neither $N$ nor $N-K$ is an integer (e.g. using $\lfloor N \rfloor + \frac{1}{2}$ and $\lfloor N-K \rfloor + \frac{1}{2}$ to replace $N$ and $N-K$.). By Lemma \ref{Perron}, with $T = T_0 \asymp N$,
\begin{align*}
& \sum_{\substack{N-K < h^2 d_1^2 d_2^2 a \leq N \\ d_i \sim D_i, h \sim H}} 1 \\ = & \frac{1}{2 \pi i} \int_{1 + \varepsilon - iT_0}^{1+ \varepsilon + iT_0} \frac{N^s - (N-K)^s}{s} \zeta(s)P(2s) ds + O(N^{\epsilon}),
\end{align*}
where
$$
P(s) = \left(\sum_{d_1 \sim D_1} \frac{1}{d_1^{s}}\right)\left(\sum_{d_2 \sim D_2} \frac{1}{d_2^{s}}\right) \left(\sum_{h \sim H} \frac{1}{h^{s}}\right).
$$
We move the line of integration to $\Re(s) = \frac{1}{2}$, The residue of
$$
\frac{N^s - (N-K)^s}{s} \zeta(s)P(2s)
$$
at $s=1$ is $O(K (H D_1 D_2)^{-1})$.

By Lemma \ref{zeta_ubd},
\begin{align*}
& \frac{1}{2 \pi i} \int_{\frac{1}{2}+ iT_0}^{1 + \varepsilon + iT_0} \frac{N^s - (N-K)^s}{s} \zeta(s)P(2s) ds\\
\ll & \max_{\frac{1}{2} \leq \sigma \leq 1 + \varepsilon}\frac{N^{\sigma}}{T_0} T_0^{\frac{1}{3}(1-\sigma)+\varepsilon} (H D_1 D_2)^{1 - 2\sigma} \ll N^{\epsilon}.
\end{align*}
Similarly,
$$
\frac{1}{2 \pi i} \int_{\frac{1}{2}- iT_0}^{1 + \varepsilon - iT_0} \frac{N^s - (N-K)^s}{s} \zeta(s)P(2s) ds \ll N^{ \epsilon}
$$
Hence the remaining task is to estimate the integral when $\Re(s)=\frac{1}{2}$. We have
\begin{align}
& \frac{1}{2 \pi i} \int_{\frac{1}{2}- iT_0}^{\frac{1}{2} + iT_0} \frac{N^s - (N-K)^s}{s} \zeta(s)P(2s) ds \notag\\
= & \frac{1}{2 \pi i} \int_{-T_0}^{T_0} \frac{N^{\frac{1}{2}+it} - (N-K)^{\frac{1}{2}+it}}{\frac{1}{2} + it} \zeta\left(\frac{1}{2} + it\right)P(1 + 2it) dt \notag \\
\ll & \int_{-T_0}^{T_0} \min\left\{KN^{-\frac{1}{2}}, \frac{N^{\frac{1}{2}}}{|t|}\right\} \left|\zeta\left(\frac{1}{2} + it\right)P(1 + 2it)\right| dt \notag\\
\ll & N^{\frac{1}{2}} \left( \int_{|t| \leq N/K} \frac{K}{N} \left|\zeta\left(\frac{1}{2} + it\right)P(1 + 2it) \right| dt + \int_{N \geq |t|> N/K} \frac{1}{|t|}\left|\zeta\left(\frac{1}{2} + it\right)P(1 + 2it)\right| dt \right) \notag \label{dyadicT}\\
\end{align}
Using a dyadic trick to deal with the second term in the bracket, we have
\begin{align*}
& \int_{|t|> N/K} \frac{1}{|t|}\left|\zeta\left(\frac{1}{2} + it\right)P(1 + 2it)\right| dt \ll \int_{N \gg |t|> N/K} \frac{1}{|t|}\left|\zeta\left(\frac{1}{2} + it\right)P(1 + 2it)\right| \frac{1}{|t|}\int_{|t|/2}^{|t|} 1 d T dt\\
\ll & \int_{N \gg T \gg N/K} \frac{1}{T^2} \int_{T}^{2T} \left|\zeta\left(\frac{1}{2} + it\right)P(1 + 2it)\right| dt dT.
\end{align*}
Hence (\ref{dyadicT}) is at most
\begin{equation}\label{dyadicTsum}
\ll N^{\frac{1}{2}} \log N \sup_{N/K \ll T \ll N} \frac{1}{T} \int_{|t| \leq T} \left|\zeta\left(\frac{1}{2} + it\right)P(1 + 2it)\right|dt.
\end{equation}
By Cauchy-Schwarz inequality and Lemmas \ref{mean_zeta} and \ref{mean_Dpoly}, we deduce that (\ref{dyadicTsum}) is
\begin{align}
\ll & N^{\frac{1}{2}}\log N \sup_{N/K \ll T \ll N} \frac{1}{T} \left(\int_{|t|\leq T} \left|\zeta(\frac{1}{2} + it)\right|^2 dt\right)^{\frac{1}{2}}\left(\int_{|t|\leq T} |P(1 + 2 it)|^2 dt\right)^{\frac{1}{2}}\notag \\
\ll & N^{\frac{1}{2}}\log N \sup_{N/K \ll T \ll N} \frac{1}{T} (T \log T)^{\frac{1}{2}} ((T+O(D_1D_2H))^{\frac{1}{2}}\left(\sum_{D_1D_2H \ll n \ll D_1D_2H}\frac{a_0u^2_3(n)}{n^2}\right)^{\frac{1}{2}}\notag \\
\ll & N^{\frac{1}{2}} (D_1 D_2 H)^{-\frac{1}{2}} \log^{C_2} N +  K^{\frac{1}{2}} \log^{C_1}N \label{uppAnalytic},
\end{align}
where $a_0u_3(n) = \sum_{n=n_1n_2n_3}1$ and $C_1,C_2>0$ are absolute constants. Hence, we have
\begin{align}
& \sum_{\substack{N-K < h^2 d_1^2 d_2^2 a \leq N \\ d_i \sim D_i, h \sim H}} 1 \notag \\
\ll & K(HD_1D_2)^{-1} + K^{\frac{1}{2}} \log^{C_1}N + N^{\frac{1}{2}} (D_1 D_2 H)^{-\frac{1}{2}} + N^{\epsilon} \log^{C_2} N + N^{\epsilon} \notag \\
\ll & K(HD_1D_2)^{-1} + Ky^{-1-\epsilon} + N^{\frac{1}{2}}y^{-\frac{1}{2}}D_2^{-\frac{1}{2}}\log^{C_2} N + N^{\epsilon} \label{dyadicUpper},
\end{align}
since $y \leq K^{\frac{1}{2} - \epsilon}$ and $D_i H \gg y$.

On the other hand, (\ref{ubdcn}) is no more than
\begin{equation}\label{TransHyperbolic}
\sum_{d_1, d_2} \sum_{N-K <n \leq N} \sum_{\substack{n = h^2 d_1^2 d_2^2 a \\ y < hd_i \leq z}} 1 \leq \sum_{d_1, d_2} \sum_{\substack{\frac{N-K}{d_1^2 d_2^2} < h^2 a \leq \frac{N}{d_1^2 d_2^2} \\ \frac{y}{\min\{d_1,d_2\}} < h \leq \frac{z}{\max\{d_1,d_2\}} }}1
\end{equation}
Without loss of generality, we can assume that $d_1 \leq d_2$, and then by the hyperbolic trick, the inner sum on the right hand side of (\ref{TransHyperbolic}) is at most
\begin{align*}
& \sum_{\frac{y}{d_1} \leq h \leq (\frac{N}{d_1^2 d_2^2})^{\frac{1}{3}} } \sum_{\frac{N-K}{h^2 d_1^2 d_2^2} \leq a \leq \frac{N}{h^2 d_1^2 d_2^2}}1
+
\sum_{a \leq (\frac{N}{d_1^2 d_2^2})^{\frac{1}{3}} } \sum_{(\frac{N-K}{a d_1^2 d_2^2})^{\frac{1}{2}} \leq h \leq (\frac{N}{a d_1^2 d_2^2})^{\frac{1}{2}}}1 \\
:= & \Sigma_1 + \Sigma_2
\end{align*}
We use the standard van der Corput method to handle $\Sigma_1$ and $\Sigma_2$ respectively. Recalling the definition of $\psi(x) = \{x\} - \frac{1}{2}$ in Lemma \ref{VDinequality}, we have
\begin{align}
\Sigma_1 & = \sum_{\frac{y}{d_1} \leq h \leq (\frac{N}{d_1^2 d_2^2})^{\frac{1}{3}} } \frac{K}{h^2 d_1^2 d_2^2} + \psi\left(\frac{N-K}{h^2 d_1^2 d_2^2}\right) - \psi\left(\frac{N}{h^2 d_1^2 d_2^2}\right) \notag \\
& \ll K y^{-1} \frac{1}{d_1 d_2^2} + O\left(\left|\sum_{\frac{y}{d_1} \leq h \leq (\frac{N}{d_1^2 d_2^2})^{\frac{1}{3}}} \psi\left(\frac{N}{h^2 d_1^2 d_2^2}\right)\right| + \left|\sum_{\frac{y}{d_1} \leq h \leq (\frac{N}{d_1^2 d_2^2})^{\frac{1}{3}}}\psi\left(\frac{N-K}{h^2 d_1^2 d_2^2}\right)\right|\right), \label{psi}
\end{align}

Let $(p,q)$ be an exponent pair (see Lemma \ref{VDinequality}) satisfying $1+2q-4p \geq 0$. By using dyadic trick and Lemma \ref{VDinequality} with $F=\frac{N}{M^2 d_1^2 d_2^2}$, we have
\begin{align*}
&\sum_{\frac{y}{d_1} \leq h \leq (\frac{N}{d_1^2 d_2^2})^{\frac{1}{3}}} \psi\left(\frac{N}{h^2 d_1^2 d_2^2}\right)  \ll \log N \max_{\frac{y}{d_1} \leq M \leq M_1 \leq 2M \leq 2(\frac{N}{d_1^2d_2^2})^{\frac{1}{3}}} \left|\sum_{h \in (M,M_1]}\psi\left(\frac{N}{h^2 d_1^2 d_2^2}\right)\right| \\
\ll & \log N \max_{\frac{y}{d_1} \leq M \leq M_1 \leq 2M \leq 2(\frac{N}{d_1^2d_2^2})^{\frac{1}{3}}}  \left(\frac{N}{d_1^2d_2^2}\right)^{\frac{p}{p+1} + \epsilon}M^{\frac{1+2q-4p}{2(p+1)}} \ll \left(\frac{N}{d_1^2d_2^2}\right)^{\frac{1+2q+2p}{6(p+1)} + \epsilon},
\end{align*}
Thus 
\begin{equation}\label{uppSigma1}
\Sigma_1 \ll Ky^{-1} \frac{1}{d_1 d_2^2} + \left(\frac{N}{d_1^2 d_2^2}\right)^{\frac{2(p+q)+1}{6(p+1)} + \epsilon}.
\end{equation}
Similarly, for all exponent pairs $\left(p,q\right)$, we have
\begin{align}
\notag & \Sigma_2 =  \sum_{a \leq (\frac{N}{d_1^2 d_2^2})^{\frac{1}{3}} } \left(\frac{N}{a d_1^2 d_2^2}\right)^{\frac{1}{2}} - \left(\frac{N-K}{a d_1^2 d_2^2}\right)^{\frac{1}{2}} +   \psi\left(\left(\frac{N-K}{a d_1^2 d_2^2}\right)^{\frac{1}{2}}\right) - \psi\left(\left(\frac{N}{a d_1^2 d_2^2}\right)^{\frac{1}{2}}\right)\\
\label{UppSigma2} \ll &  K N^{- \frac{1}{3}} \frac{1}{(d_1 d_2)^{\frac{4}{3}}} + O\left(\left|\sum_{a \leq (\frac{N}{d_1^2 d_2^2})^{\frac{1}{3}}} \psi\left(\left(\frac{N}{a d_1^2 d_2^2}\right)^{\frac{1}{2}}\right)\right| + \left|\sum_{a \leq (\frac{N}{d_1^2 d_2^2})^{\frac{1}{3}}}\psi\left(\left(\frac{N-K}{a d_1^2 d_2^2}\right)^{\frac{1}{2}}\right)\right|\right)
\end{align}
As above, we split $a$ to dyadic ranges $a \sim A$ and use Lemma \ref{VDinequality} with $M = A$ and $F = (\frac{N}{A d_1^2 d_2^2})^{1/2}$. We obtain
\begin{equation}\label{uppSigma2}
\Sigma_2 \ll K N^{- \frac{1}{3}} \frac{1}{(d_1 d_2)^{\frac{4}{3}}} + \left(\frac{N}{d_1^2 d_2^2}\right)^{\frac{2(p+q)+1}{6(p+1)} + \epsilon} 
\end{equation}
Let $\sideset{}{^\flat}{\sum}$ represent the sum over powers of two and let $D$ be a parameter to be chosen later.
Hence by (\ref{dyadicUpper}), (\ref{TransHyperbolic}), (\ref{uppSigma1}) and (\ref{uppSigma2}) we get the following upper bound for (\ref{ubdcn})
\begin{align}
& \sum_{\substack{N-K \leq h^2 d_1^2 d_2^2 a \leq N \\ y \leq hd_i \leq z} } 1 \ll \sideset{}{^\flat}{\sum}_{D_2 \leq D} \sum_{\substack{N-K \leq h^2 d_1^2 d_2^2 a \leq N \\ y \leq hd_i \leq z \\ d_2 \sim D_2 \\ d_1 \leq d_2}} 1 +  \sideset{}{^\flat}{\sum}_{D_2 > D} \sum_{\substack{N-K \leq h^2 d_1^2 d_2^2 a \leq N \\ y \leq hd_i \leq z \\ d_2 \sim D_2 \\ d_1 \leq d_2}} 1 \notag\\
& \ll \sideset{}{^\flat}{\sum}_{D_2 \leq D }  \sum_{\substack{d_2 \sim D_2 \\ d_1 \leq d_2}} \sum_{\substack{\frac{N-K}{d_1^2 d_2^2} \leq h^2 a \leq \frac{N}{d_1^2 d_2^2} \\ \frac{y}{d_1} \leq h \leq \frac{z}{d_2} }}1 + \sideset{}{^\flat}{\sum}_{\substack{D_2 \gg D \\ D_1 \ll D_2 \\ y \ll HD_i \ll z }} \left(K(HD_1D_2)^{-1} + K y^{-1-\epsilon} + N^{\frac{1}{2} + \epsilon}y^{-\frac{1}{2}} D_2^{-\frac{1}{2}} + N^{\epsilon}\right) \notag\\
\label{uppVander}& \ll   Ky^{-1} + Ky^{-1-\epsilon}\log^2 N + N^{\frac{2(p+q)+1}{6(p+1)}+\epsilon}D^{2-\frac{4p+4q+2}{3(p+1)}} + N^{\frac{1}{2} + \epsilon}y^{-\frac{1}{2}} D^{-\frac{1}{2}} +\sideset{}{^\flat}{\sum}_{\substack{D_2 \gg D \\ D_1 \ll D_2 \\ y \ll HD_i \ll z }} K(HD_1D_2)^{-1}.
\end{align}

We first calculate $\sideset{}{^\flat}{\sum}_{\substack{D_2 \gg D \\ D_1 \ll D_2 \\ y \ll HD_i \ll z }} K(HD_1D_2)^{-1}$ in (\ref{uppVander}),  writing $D_1 = 2^{n_1}, D_2 = 2^{n_2}$ and $H = 2^{h_0}$, which is no more than
\begin{align}
\notag & K \sum_{n_2 \geq 0} \sum_{n_1 \leq n_2 + O(1)} \sum_{h_0 \geq \log_2(\frac{y}{2^{n_1}}) -O(1)} (2^{n_1}2^{n_2}2^{h_0})^{-1} \\
\label{KHD1D2-1}\ll & K y^{-1}  \sum_{n_2 \geq 0} 2^{-n_2}\sum_{n_1 \leq n_2 + O(1)} 1 \ll Ky^{-1}
\end{align}
Then we choose the exponent pair $(p,q)=(\frac{2}{7}, \frac{1}{14})$, see \cite[(8.16)-(8.17)]{IK04}, and $D = N^{\frac{5}{29} + \epsilon}y^{-\frac{9}{29}}$, which yields that
\begin{equation}\label{balance}
N^{\frac{2(p+q)+1}{6(p+1)}+\epsilon}D^{2-\frac{4p+4q+2}{3(p+1)}} + N^{1/2 + \epsilon}y^{-1/2} D^{-\frac{1}{2}} \ll N^{\frac{12}{29} + \epsilon}y^{-\frac{10}{29}}.
\end{equation}
Hence, the claim follows from (\ref{ubdcn}), (\ref{uppVander}), (\ref{KHD1D2-1}) and (\ref{balance}).
\end{proof}
We can easily get the following lemma by using Lemma \ref{key}.
\begin{lemma}\label{Balance}
 Let $N > K \geq 1$, $y=K^{\frac{1}{r+1}}$ and $z=N^{\frac{1}{r}}$. Then
$$
\sum_{N-K < n \leq N}|c_n(y,z)|^2 \ll K^{\frac{2}{r+1}}
$$
provided that
$$
K \gg
\begin{cases}
N^{\frac{9}{17} + \epsilon} & \text{if } r=2,\\
N^{\frac{r+1}{2r} + \epsilon} & \text{if } r \geq 3.
\end{cases}
$$
\end{lemma}

\begin{Rem}
The proof of  Lemma \ref{key} is motivated by the methods used in dealing with the Dirichlet divisor problem, see e.g. \cite{JS14}. In the above proof, we combine analytic methods with hyperbolic trick and van der Corput bounds. If we only use the analytic approach, we can also get a non-trivial result but the length $K$ must be longer than $N^{\frac{3}{5}+ \epsilon}$ in Lemma \ref{Balance}, which is weaker than our result. On the other hand, if we just use the hyperbolic trick and van der Corput bound, we cannot get any non-trivial improvement. The exponent $\frac{9}{17}$ is determined by the exponent pair $(p,q)=(\frac{2}{7}, \frac{1}{14})$. One could use a stronger exponent pair to obtain a similar result for slightly smaller $K$, but the improvement would not be very significant.
\end{Rem}
\section{Proof of the lower bound results}\label{pflow}
\setcounter{lemma}{0} \setcounter{theorem}{0}
\setcounter{equation}{0}
In this section we will prove the lower bound $\|F_H*g_1\|_{L_1} \gg H^\frac{1}{r+1}$ which, by the argument in the end of Section \ref{intro}, implies Theorem \ref{sqfre} and Theorem \ref{Mobius}. We first define a ``$q$-analog'' of $F_{H}$ as
$$
F_H^q(\alpha) := \frac{1}{q} \sum_{a=1}^{q} F_{H}\left(\alpha - \frac{a}{q}\right)= \sum_{\substack{|n-N| \leq H \\ n \equiv 0 \pmod{q}}}\left(1 - \frac{|n-N|}{H}\right) e(n \alpha).
$$
Let $y = H^{\frac{1}{r+1}}$ and $z = N^{\frac{1}{r}}$. By (\ref{FourCov}) and (\ref{FEx}), we have
\begin{align}
& F_H*g_1(\alpha) = \sum_{|n - N| \leq H} \left(1 - \frac{|n-N|}{H}\right) a_n e(n \alpha). \notag \\
= & \sum_{|n - N| \leq H} \left(1 - \frac{|n-N|}{H}\right) c_n(1,y)e(n \alpha)\notag \\
+ & \sum_{|n - N| \leq H} \left(1 - \frac{|n-N|}{H}\right) c_n(y,z)e(n \alpha) \notag \\
= & \sum_{d_0 \leq y} \mu(d_0) F_{H}^{d_0^r}(\alpha) + \sum_{|n - N| \leq H} \left(1 - \frac{|n-N|}{H}\right) c_n(y,z)e(n \alpha). \label{FHg1}
\end{align}
For the second term, we will use Lemma \ref{key} to handle it. Now we follow Balog-Rusza's argument to deal with the first term. We denote the first term of (\ref{FHg1}) by $g_2(\alpha)$ and transform it to
\begin{align*}
 g_2(\alpha) & = \sum_{d_0 \leq y} \frac{\mu(d_0)}{d_0^r} \sum_{a_0=1}^{d_0^r}F_H\left(\alpha - \frac{a_0}{d_0^r}\right)\\
 & = \sum_{d_0 \leq y}  \sum_{m|d_0} \sum_{\substack{a_0=1 \\ m^r \mid a_0 \\ (a_0/m^r,d_0^r/m^r)r\text{-free}}}^{d_0^r} \frac{\mu(d_0)}{d_0^r} F_H\left(\alpha - \frac{a_0}{d_0^r}\right).\\
\end{align*}
Writing $d_0 = md$ and $a_0 = m^r a$ and noticing that $\mu(d_0) = \mu(m)\mu(d)\1_{(m,d)=1}$, we obtain
\begin{align}
\notag g_2(\alpha) & = \sum_{d \leq y} \frac{\mu(d)}{d^r} \sum_{\substack{m \leq y/d \\ (m,d)=1}} \frac{\mu(m)}{m^r} \sum_{\substack{a=1 \\ (a,d^r)r\text{-free}}}^{d^r}F_H\left(\alpha - \frac{a}{d^r}\right)\\
\label{mubG} & = \sum_{d \leq y} \mu(d)b_dG_d(\alpha),
\end{align}
where
$$
b_d = \sum_{\substack{m \leq y/d \\ (m,d)=1}} \frac{\mu(m)}{m^r},
$$
and
\begin{equation}\label{DefGd}
G_d(\alpha) = \frac{1}{d^r} \sum_{\substack{a=1 \\ (a,d^r)r\text{-free}}}^{d^r}F_H\left(\alpha - \frac{a}{d^r}\right).
\end{equation}
It is clear that
\begin{equation}\label{sizebd}
\frac{1}{3} \leq 1 -(\frac{\pi^2}{6} - 1) \leq 1 - \sum_{d=2}^{\infty} d^{-r} \leq b_d \leq 1 + \sum_{d=2}^{\infty}d^{-2} \leq \frac{\pi^2}{6} \leq \frac{5}{3}.
\end{equation}
and by (\ref{Flow}) we have
\begin{equation}\label{FHbeta}
|F_{H}(\beta)| \geq \frac{H}{2},\quad \text{ whenever } \quad \|\beta\| \leq \frac{1}{2H}.
\end{equation}
To show that $|F_H|$ is large, we must define a set with positive density in $\T$, in which $|F_H|$ is large. For any $d \leq y$, we define the set
$$
\cX_d=\bigcup_{\substack{1 \leq a \leq d^r \\ (a,d^r)r\text{-free}}}\left[\frac{a}{d^r}-\frac{1}{2H}, \frac{a}{d^r}+\frac{1}{2H}\right].
$$

For any fixed $d \leq y$ and two distinct $1 \leq a_1, a_2 \leq d^r$ the distance between $\frac{a_1}{d^r}$ and ${\frac{a_2}{d^r}}$ is at least
\begin{equation}\label{dyH}
\frac{1}{d^r} \geq \frac{1}{y^r} = \frac{1}{H^{\frac{r}{r+1}}}.
\end{equation}
Thus we have
\begin{equation}\label{disXd}
\left[\frac{a_1}{d^r}-\frac{1}{2H}, \frac{a_1}{d^r}+\frac{1}{2H}\right] \bigcap \left[\frac{a_2}{d^r}-\frac{1}{2H}, \frac{a_2}{d^r}+\frac{1}{2H}\right] = \emptyset.
\end{equation}
In Balog-Rusza's argument, their $F(\alpha)$ is non-negative, so they can just pick up one large term $F(\alpha - \frac{a}{d^r})$ in the sum of $G_d(\alpha)$ directly as the lower bound of $G_d(\alpha)$. However, since $F_H(\alpha)$ is not always non-negative, we must be more careful.

\begin{lemma}\label{sizeGd}
Let $1 \leq d \leq y$. For any $\alpha \in \cX_d$, we have that
$$
|G_d(\alpha)| \geq \frac{1}{2}(1+o(1)) \frac{H}{d^r}.
$$
\end{lemma}

\begin{proof}
For any $\alpha \in \cX_d$, by (\ref{disXd}) there is an unique $a$ such that $\alpha \in [\frac{a}{d^r}-\frac{1}{2H}, \frac{a}{d^r}+\frac{1}{2H}]$. Thus by (\ref{FHbeta}), (\ref{DefGd}), (\ref{Fupp}) and (\ref{dyH}), we have, for any $d \leq y$ and $\alpha \in \cX_d$,
\begin{align}
\notag |G_d(\alpha)| & \geq \frac{H}{2 d^r} - \frac{1}{d^r} \sum_{\substack{1 \leq a_1 \leq d^r \\ (a_1,d^r)r\text{-free} \\ a_1 \neq a}} \left|F_{H}\left(\alpha - \frac{a_1}{d^r}\right)\right| = \frac{H}{2 d^r} - \frac{1}{d^r}\sum_{\substack{1 \leq a_1 \leq d^r \\ (a_1,d^r)r\text{-free} \\ a_1 \neq a}} \frac{1}{H \|\alpha - \frac{a_1}{d^r}\|^2}\\
\notag & \geq \frac{H}{2 d^r} - \frac{1}{d^r}\sum_{\substack{1 \leq a_1 \leq d^r \\ (a_1,d^r)r\text{-free} \\ a_1 \neq a}} \frac{1}{H (\| \frac{a - a_1}{d^r}\| - \frac{1}{2H})^2} \geq \frac{H}{2 d^r} - \frac{1}{d^r}\sum_{k=1}^{d^r/2} \frac{2}{H (\frac{k}{d^r}-\frac{1}{2H})^2}  \\
\label{lowerGdalpha}& \geq \frac{H - O\left(H^{\frac{2r}{r+1}-1}\right)}{2 d^r} \geq \frac{1}{2}(1+o(1)) \frac{H}{d^r}
\end{align}
\end{proof}

\begin{Rem}
To deal with $\sum_{\substack{1 \leq a_1 \leq d^r \\ (a_1,d^r)r\text{-free} \\ a_1 \neq a}} |F(\alpha - \frac{a_1}{d^r})|$, one might consider to use the Cauchy-Schwarz inequality and then use the large sieve inequality. However, this method may fail.
\end{Rem}
Since the sets $\cX_d$ are not necessarily disjoint for different $d \leq y$, we cannot directly calculate the size of $\bigcup_{d} \cX_d$. In order to overcome this obstacle, we define pairwise disjoint $\cY_d$ which have the similar size as $\cX_d$. We define
\begin{equation}\label{DefYd}
\cY_d := \left\{\alpha \in \cX_d: \sum_{\substack{d' \leq y \\ d' \neq d}}|G_{d'}(\alpha)| \leq \frac{H}{20d^r} \right\}.
\end{equation}

\begin{lemma}\label{disYd}
The sets $\cY_d$ defined in (\ref{DefYd}), for $d \leq y$, are pairwise disjoint.
\end{lemma}

\begin{proof}
Suppose that $1 \leq d_1 < d_2 \leq y$ such that $\alpha \in \cY_{d_1} \cap \cY_{d_2} \neq \emptyset$. By Lemma \ref{sizeGd} and (\ref{DefYd}), we have
$$
|G_{d_1}(\alpha)| \geq (1+o(1))\frac{H}{2d_1^r} \quad \text{  and  } \quad |G_{d_1}(\alpha)| \leq \frac{H}{20d_2^r},
$$
which contradict each other.
\end{proof}
In case $\alpha \in \cY_d$, (\ref{mubG}), (\ref{lowerGdalpha}) and (\ref{sizebd}) imply that
\begin{equation}\label{lowg2}
|g_2(\alpha)| \geq b_d |G_d(\alpha)| - \sum_{\substack{d' \leq y \\ d' \neq d}}b_{d'}|G_{d'}(\alpha)| \gg \frac{H}{d^r}.
\end{equation}
Now we estimate the sizes of $\cX_d$ and $\cY_d$. First, it is easy to see that
$$
|\cX_d| =\frac{1}{H} \sum_{\substack{a=1 \\ (a,d^r)r\text{-free}}}^{d^r}1 = \frac{d^r}{H} \prod_{p|d}\left(1-\frac{1}{p^r}\right),
$$
and therefore
\begin{equation}\label{sizeXd}
\frac{d^r}{H} \ll |\cX_d| \leq \frac{d^r}{H}.
\end{equation}

Next, we estimate the size of $\cZ_d = \cX_d \setminus \cY_d$. If $\alpha \in \cZ_d$ then
$$
|G_d(\alpha)| \geq (1+o(1))\frac{H}{2 d^r} \quad \text{  and  } \quad \sum_{\substack{d'\leq y \\ d'\neq d}}|G_{d'}(\alpha)| \geq \frac{H}{20d^r}.
$$
This immediately implies that
\begin{equation}\label{upperZ}
\int_{\cZ_d} \left|G_d(\alpha)\right|\sum_{\substack{d'\leq y \\ d'\neq d}}|G_{d'}(\alpha)| d \alpha \geq |\cZ_d| \frac{H^2}{40d^{2r}}.
\end{equation}

Now we follow the argument of Balog-Ruzsa to prove a ``quasi-orthogonality'' property of $|G_d(\alpha)|$. Namely,
\begin{lemma}[quasi-orthogonality]\label{quasi}
For any $1 \leq d_1 < d_2 \leq y$, we have
$$
\int_{\T} |G_{d_1}(\alpha)||G_{d_2}(\alpha)| d \alpha \ll 1.
$$
\end{lemma}
\begin{proof}
First by (\ref{Fupp}) we observe that
$$
\|F_{H}\|_{L_1} = \int_{\T}|F_{H}(\alpha)| d \alpha  \ll 1.
$$
By (\ref{Fupp}), for any $\alpha \in \T$, we have
\begin{align}
& \int_{\T}|F_H(\beta)||F_{H}(\alpha + \beta)| d \beta\\
\leq & \int_{\T}(|F_H(\beta)|+|F_{H}(\alpha + \beta)|) \min \{|F_H(\beta)|, |F_H(\alpha + \beta)|\} d \beta \notag\\
\leq & \int_{\T}(|F_H(\beta)|+|F_{H}(\alpha + \beta)|) \min \left\{H, \frac{1}{H \|\beta\|^2}, \frac{1}{H \|\alpha+\beta\|^2}\right\} d \beta \notag\\
\ll & \int_{\T}(|F_H(\beta)|+|F_{H}(\alpha + \beta)|) \min \left\{H, \frac{1}{H (\|\beta\| + \|\beta +\alpha\|)^2} \right\} d \beta \notag\\
\leq & \int_{\T}(|F_H(\beta)|+|F_{H}(\alpha + \beta)|) \min \left\{H, \frac{1}{H \|\alpha\|^2} \right\} d \beta \notag\\
\leq & \min \left\{H, \frac{1}{H \|\alpha\|^2} \right\} \label{corF}
\end{align}
By (\ref{DefGd}) and (\ref{corF}) we have
\begin{align}
&\int_{\T} |G_{d_1}(\alpha)||G_{d_2}(\alpha)| d \alpha \notag\\
\leq &\frac{1}{d_1^r d_2^r} \sum_{\substack{a_1=1 \\ (a_1,d_1^r)r\text{-free}}}^{d_1^r} \sum_{\substack{a_2=1 \notag\\ (a_2,d_2^r)r\text{-free}}}^{d_2^r}\int_{\T}\left|F_H(\alpha - \frac{a_1}{d_1^r})F_H(\alpha - \frac{a_2}{d_2^r})\right| d \alpha \notag\\
\ll & \frac{1}{d_1^r d_2^r} \sum_{\substack{a_1=1 \\ (a_1,d_1^r)r\text{-free}}}^{d_1^r} \sum_{\substack{a_2=1 \\ (a_2,d_2^r)r\text{-free}}}^{d_2^r} \min\left\{H, \frac{1}{H \|\frac{a_1}{d_1^r} - \frac{a_2}{d_2^r}\|^2}\right\} \label{lastquasi}
\end{align}
We can write
$$
\left\|\frac{a_1}{d_1^r} - \frac{a_2}{d_2^r}\right\| = \frac{|m|}{[d_1^r,d_2^r]},
$$
where $m$ is the member of the residue class
$$
a_1 \frac{d_2^r}{(d_1^r,d_2^r)} - a_2\frac{d_1^r}{(d_1^r,d_2^r)} \pmod{[d_1^r,d_2^r]}
$$
with least absolute value. Note that $m=0$ does not appear because $d_1 \neq d_2$ and $(a_1,d_1^r),(a_2,d_2^r)$ are $r$-free. Given a non-zero $|m| \leq [d_1^r,d_2^r]/2$ the above holds when
$$
a_1 \frac{d_2^r}{(d_1^r,d_2^r)} \equiv m \pmod{\frac{d_1^r}{(d_1^r,d_2^r)}},
$$
which happens for exactly $(d_1^r,d_2^r)$ choices of $a_1 \pmod{d_1^r}$. When $m$ and $a_1$ are fixed, $a_2$ is uniquely determined. In (\ref{lastquasi}), we get that
\begin{align*}
& \int_{\T} |G_{d_1}(\alpha)||G_{d_2}(\alpha)| d \alpha \ll \frac{(d^r_1,d^r_2)}{d_1^rd_2^r} \sum_{1 \leq m \leq [d_1^r,d_2^r]/2} \min\left\{H, \frac{[d_1^r,d_2^r]^2}{Hm^2}\right\}\\
\ll & \frac{H}{[d_1^r, d_2^r]}\sum_{1 \leq m \leq [d_1^r,d_2^r]/H} 1 + \frac{[d_1^r,d_2^r]}{H}\sum_{m > [d_1^r,d_2^r]/H} \frac{1}{m^2} \ll 1.
\end{align*}
\end{proof}
By (\ref{upperZ}) and Lemma \ref{quasi}, we have
\begin{equation}\label{sizeZd}
|\cZ_d| \ll \frac{d^{2r}y}{H^2}
\end{equation}
Let $\epsilon > 0$ be a small fixed constant. By (\ref{sizeXd}) and (\ref{sizeZd}) we have
\begin{equation}\label{sizeYd}
|\cY_d| = |\cX_d| - |\cZ_d|\gg \frac{d^r}{H} \gg |\cX_d|, \quad \text{ for  every } \quad d \leq \epsilon y.
\end{equation}
Let
$$
\cY = \bigcup_{d \leq \epsilon y} \cY_d.
$$
We have
\begin{equation}\label{sizeY}
|\cY| \leq \sum_{d \leq \epsilon y} |\cX_d| \leq \frac{(\epsilon y)^{r+1}}{H} = \epsilon^{r+1}.
\end{equation}

By (\ref{FHg1}), Cauchy-Schwarz inequality and Parseval's identity, we have
\begin{align}
\notag & \|F*g_1\|_{L_1} = \int_{\T} |F*g_1(\alpha)| d \alpha \geq \int_{\cY} |F*g_1(\alpha)| d \alpha\\
\label {Ycn2}\geq & \int_{\cY}|g_2(\alpha)| d \alpha - \left(|\cY| \sum_{|n-N| \leq H} |c_n(y,z)|^2\right)^{\frac{1}{2}}.
\end{align}
By (\ref{lowg2}), (\ref{sizeYd}) and Lemma \ref{disYd},
$$
\int_{\cY}|g_2(\alpha)| d \alpha \gg  \sum_{d \leq \epsilon y} \frac{H}{d^r} |\cY_d| \gg \epsilon y \gg \epsilon H^{\frac{1}{r+1}},
$$
and by Lemma \ref{Balance} and (\ref{sizeY})
$$
\left(|\cY| \sum_{|n-N| \leq H} |c_n(y,z)|^2\right)^{\frac{1}{2}} \ll \epsilon^{\frac{r+1}{2}}H^{\frac{1}{r+1}},
$$
when $H$ satisfies the condition of $K$ in Lemma \ref{Balance}.
Combining these with (\ref{Ycn2}) implies, when $\epsilon > 0$ is small enough, $\|F_H*g_1\|_{L_1} \gg H^\frac{1}{r+1}$. As explained in the beginning of this section, this completes the proof of Theorem \ref{sqfre} and the proof of the lower bound case of Theorem \ref{rfree}.

\section{Proof of the upper bound case of theorem \ref{rfree}}\label{pfupp}
\setcounter{lemma}{0} \setcounter{theorem}{0}
\setcounter{equation}{0}
In this section, we will prove the upper for the $L_1$ norm of the exponential sum over $r$-free numbers. This proof does not involve $F_{H}$. Instead we involve difference of two Fej\'er Kernels, see (\ref{FejDifFour}), and the argument is essentially same as Balog-Ruzsa's.

Let $D = H^{\frac{1}{r+1}}$, and let $S$ be such that $2^{S-1} < D \leq 2^S $. Note that
$$
S \leq \frac{1}{r+1}\log_2 H + 1.
$$
Recall (\ref{cnyz}) and decompose $a_n$ as follows
$$
a_n = \sum_{s=1}^{S} c_n(D2^{-s}, D2^{-(s-1)}) + c_{n}(D, N^{\frac{1}{r}}).
$$
We have
\begin{align}
\notag& \int_{\T} \left|\sum_{|n-N| \leq H} a_n e(n \alpha) \right| d \alpha \\
\label{anupp}\leq & \sum_{s=1}^S \int_{\T} \left|\sum_{|n-N| \leq H} c_n(D2^{-s}, D2^{-(s-1)}) e(n \alpha)\right| d \alpha + \int_{\T} \left|\sum_{|n-N|\leq H}c_{n}(D, N^{\frac{1}{r}})e(n \alpha)\right| d \alpha.
\end{align}
The second term can be estimated as in the proof of the lower bound case: The Cauchy-Schwarz inequality, Parseval's identity and Lemma \ref{Balance} imply that
\begin{equation}\label{cntailupp}
\int_{\T} \left|\sum_{|n-N|\leq H}c_{n}(D, N^{1/r})e(n \alpha)\right| d \alpha \ll H^{1/(r+1)}.
\end{equation}
The remaining task is to estimate the first term. Let $K = H 2^{-rs}$ and rewrite
\begin{align*}
& \sum_{|n-N| \leq H} c_n(D2^{-s}, D2^{-(s-1)}) e(n \alpha) = \sum_{|n|\leq H}c_{n+N}(D2^{-s}, D2^{-(s-1)}) e((n+N) \alpha)\\
= & \sum_{|n| \leq H+K} \min\left\{1, \frac{H+K-|n|}{K}\right\} c_{N+n}(D2^{-s}, D2^{-(s-1)})  e((n+N) \alpha) \\
- & \sum_{H < |n| \leq H+K} \frac{H + K -|n|}{K} c_{n+N}(D2^{-s}, D2^{-(s-1)})  e((n+N) \alpha)\\
=: & \Sigma_3(s) + \Sigma_4(s).
\end{align*}
Before estimating each term, let us explain the motivation of this decomposition. The coefficients in $\Sigma_3(s)$ comes from the difference of two Fej\'er Kernels. By involving this difference of Fej\'er Kernels we can make the coefficient ``smooth'', which is useful in the later integral. However, the payoff is that we should add the ``tails''--$\Sigma_4(s)$ from $H$ to $H+K$. Fortunately, we can control the ``tails'' using Lemma \ref{key}.

Let us estimate $\int_{\T} \Sigma_4(s)$ first. In fact by the Cauchy-Schwarz inequality and Parseval's identity
\begin{align}
\notag & \int_{\T} |\Sigma_4(s)| d \alpha \leq \left(\sum_{H < |n| \leq H+K}|c_{n+N}(D2^{-s}, D2^{-(s-1)})|^2\right)^{\frac{1}{2}}\\
\label{cnD2s2s-1}= & \left(\sum_{N+H< n \leq N+H+K} + \sum_{N-H-K \leq n <N-H} |c_{n}(D2^{-s}, D2^{-(s-1)})|^2\right)^{\frac{1}{2}}.
\end{align}
For $r=2$, we have
$$
D2^{-s} = H^{\frac{1}{3}}2^{-s} \leq H^{\frac{1}{2}-\epsilon}2^{-s} \leq (H2^{-2s})^{\frac{1}{2} - \epsilon} = K^{\frac{1}{2} - \epsilon},
$$
which implies that (\ref{cnD2s2s-1}) satisfies the condition of Lemma \ref{key}. By Lemma \ref{key}, (\ref{cnD2s2s-1}) is
\begin{align*}
\ll&
\begin{cases}
(K(D 2^{-s})^{1-r} + N^{\frac{12}{29}+\epsilon} )^{\frac{1}{2}} & r=2\\
(K(D 2^{-s})^{1-r} + N^{\frac{1}{r} + \epsilon} )^{\frac{1}{2}} & \text{otherwise}
\end{cases}\\
\ll &
\begin{cases}
H^{\frac{1}{3}}2^{-\frac{s}{2}}+N^{\frac{6}{29} + \epsilon} & r=2\\
H^{\frac{1}{r+1}}2^{-\frac{s}{2}} + N^{\frac{1}{2r} + \epsilon} & \text{otherwise}
\end{cases}
\end{align*}
Thus when $H$ satisfies (\ref{lengthHrfree}), we have
\begin{equation}\label{Sigma4upp}
\sum_{s = 1}^S \int_{\T} |\Sigma_4(s)| d \alpha \ll H^{\frac{1}{r+1}}
\end{equation}
The rest argument for estimating the contribution of the $\Sigma_3(s)$ part is same as Balog-Ruzsa's (2.11), but to make paper self-contained, we write down it.

\begin{lemma}
For any $ H \geq 0, K \geq 1, 1 \leq d \leq H+K $ and $M$, we have
$$
\sum_{\substack{|n|\leq H+K \\ n \equiv M \pmod{d}}} \min\left\{1, \frac{H+K-|n|}{K}\right\}e(n \alpha) \ll \min\left\{\frac{H+K}{d}, \frac{1}{\|d \alpha\|} , \frac{d}{K\|d \alpha\|^2}\right\}
$$
\end{lemma}
\begin{proof}
See \cite[(2.5)]{BR01}.
\end{proof}

Recall the definition of $c_n$ in (\ref{cnyz}). By the above lemma, we have
\begin{align*}
& \int_{\T} |\Sigma_3(s)| d \alpha\\
 = &\int_{\T} \left| \sum_{|n|\leq H+K} \min \left\{1, \frac{H+K-|n|}{K}\right\}\sum_{\substack{d^r \mid (n+N) \\ D2^{-s} < d \leq D2^{-(s-1)}}}\mu(d)e((n+N)\alpha) \right| d \alpha\\
 \ll & \sum_{D2^{-s} < d \leq D2^{-(s-1)}}\int_{\T} \left| \sum_{\substack{|n|\leq H+K \\ n \equiv -N \pmod{d^r}}} \min \left\{1, \frac{H+K-|n|}{K}\right\}e(n \alpha) \right| d \alpha\\
\ll &  \sum_{D2^{-s} < d \leq D2^{-(s-1)}}\int_{\T}\min\left\{\frac{H+K}{d^r}, \frac{1}{\|d^r \alpha\|} , \frac{d^r}{K\|d^r \alpha\|^2}\right\} d \alpha\\
\ll &  \sum_{D2^{-s} < d \leq D2^{-(s-1)}}\int_{\T}\min\left\{\frac{H}{d^r}, \frac{1}{\|\alpha\|} , \frac{d^r}{K\| \alpha\|^2}\right\} d \alpha\\
\ll & \sum_{D2^{-s} < d \leq D2^{-(s-1)}} \left(\int_{0}^{\frac{d^r}{H}}\frac{H}{d^r} d \alpha +\int_{\frac{d^r}{H}}^{\frac{d^r}{K}}\frac{1}{\alpha}d \alpha + \int_{\frac{d^r}{K}}^{\frac{1}{2}}\frac{d^r}{K \alpha^2}d \alpha\right)\\
\ll & \sum_{D2^{-s} < d \leq D2^{-(s-1)}} \left(1 + \log \frac{H}{K}\right) \ll s 2^{-s}D.
\end{align*}
So it is easy to see that
\begin{equation}\label{Sigma3upp}
\sum_{s=1}^S \int_{\T} |\Sigma_3(s)| d \alpha \ll H^{\frac{1}{r+1}}
\end{equation}
Hence, the upper bound case of Theorem \ref{rfree} follows from (\ref{anupp}), (\ref{cntailupp}), (\ref{Sigma3upp}) and (\ref{Sigma4upp}).

\end{document}